\documentclass[11pt,reqno,oneside]{amsart}

\usepackage[asymmetric,top=3.5cm,bottom=4.3cm,left=3.1cm,right=3.1cm]{geometry}
\usepackage{amsmath,amsfonts,amsthm,mathrsfs,amssymb,cite}
\usepackage[usenames]{color}
\usepackage{color}
\usepackage{graphicx}
\usepackage{mathtools} 
\usepackage{graphics}
\usepackage{framed}
\usepackage{exscale}
\usepackage{relsize}

\usepackage{enumerate}
\usepackage{epstopdf}
\usepackage{graphicx}
\usepackage{latexsym}
\usepackage{enumerate}

\theoremstyle{plain}
\newtheorem{theorem}{Theorem}[section]

\newtheorem{lemma}{Lemma}[section]

\newtheorem{definition}{Definition}[section]

\newtheorem{remark}{Remark}[section]

\renewcommand{\eqref}[1]{\textnormal{(\ref{#1})}}

\numberwithin{equation}{section}

\begin{document}
\title[Boundary localization of transmission eigenfunctions]{Boundary localization of transmission eigenfunctions in spherically stratified media}

\author{Yan Jiang}
\address{Department of Mathematics, Jilin University, Changchun, Jilin, China.}
\email{jiangyan20@mails.jlu.edu.cn}

\author{Hongyu Liu}
\address{Department of Mathematics, City University of Hong Kong, Hong Kong SAR, China.}
\email{hongyu.liuip@gmail.com, hongyliu@cityu.edu.hk}

\author{Jiachuan Zhang}
\address{School of Physical and Mathematical Sciences, Nanjing Tech University, Nanjing, Jiangsu, China.}
\email{zhangjc@njtech.edu.cn}

\author{Kai Zhang}
\address{Department of Mathematics, Jilin University, Changchun, Jilin, China.}
\email{zhangkaimath@jlu.edu.cn}


\begin{abstract}

Consider the transmission eigenvalue problem for $u \in H^1(\Omega)$ and $v\in H^1(\Omega)$:
\[
\left\{
\begin{array}{ll}
\nabla \cdot (\sigma \nabla u)+k^2 \mathbf{n}^2u =0  & \text{in} \ \ \Omega, \medskip \\
\Delta v + k^2 v =0  & \text{in} \ \ \Omega, \medskip \\
\displaystyle u=v, \  \sigma\frac{\partial u}{\partial \nu}=\frac{\partial v}{\partial \nu} & \text{on} \ \ \partial\Omega, \medskip \\
\end{array}
\right.
\]
where $\Omega$ is a ball in $\mathbb{R}^N$, $N=2,3$. If $\sigma$ and $\mathbf{n}$ are both radially symmetric, namely they are functions of the radial parameter $r$ only, we show that there exists a sequence of transmission eigenfunctions $\{u_m, v_m\}_{m\in\mathbb{N}}$ associated with $k_m\rightarrow+\infty$ as $m\rightarrow+\infty$ such that the $L^2$-energies of $v_m$'s are concentrated around $\partial\Omega$. If $\sigma$ and $\mathbf{n}$ are both constant, we show the existence of transmission eigenfunctions $\{u_j, v_j\}_{j\in\mathbb{N}}$ such that both $u_j$ and $v_j$ are localized around $\partial\Omega$. Our results extend the recent studies in \cite{CDHLW,CDLS}. Through numerics, we also discuss the effects of the medium parameters, namely $\sigma$ and $\mathbf{n}$, on the geometric patterns of the transmission eigenfunctions.

\medskip

\noindent{\bf Keywords:}~~ Transmission eigenfunctions, spectral geometry, boundary localization, wave localization 

\noindent{\bf 2010 Mathematics Subject Classification:}~~35P25, 78A46 (primary); 35Q60, 78A05 (secondary).

\end{abstract}

\maketitle

\section{Introduction}

In this paper, we study the geometric patterns of transmission eigenfunctions. To begin with, we briefly discuss the physical origin of the transmission eigenvalue problem.

Let us consider the time-harmonic wave scattering caused by the interaction of an incident wave field and an inhomogeneous medium. Let $u^i$ denote the incident field, which is an entire solution to $(\Delta+k^2)u^i=0$ in $\mathbb{R}^N$, $N=2,3$. Here, $k\in\mathbb{R}_+$ signifies the (normalized) angular frequency of the wave propagation. Let $(\Omega;\sigma, \mathbf{n})$ signify the inhomogeneous medium. Here, $\Omega$ denotes the support of the inhomogeneity of the medium, which is a bounded Lipschitz domain such that $\mathbb{R}^N\backslash\overline{\Omega}$. $\sigma$ and $\mathbf{n}$ signify the medium parameters. It is assumed that $\sigma$ and $\mathbf{n}$ are $L^\infty$ functions and both are bounded below by a positive constant. We also let $\sigma=\mathbf{n}\equiv 1$ in $\mathbb{R}^N\backslash\overline{\Omega}$. Let $u$ and $u^s:=u-u^i$, respectively, denote the total and scattered wave fields. The wave scattering is governed by the following Helmholtz system:
\begin{equation}\label{eq:Helm1}
\nabla\cdot(\sigma\nabla u)+k^2 \mathbf{n}^2 u=0\quad\mbox{in}\ \mathbb{R}^N;\ \ \lim_{r\rightarrow+\infty} r^{(N-1)/2}\left(\partial_r u^s-\mathrm{i}k u^s\right)=0,
\end{equation}
where $\mathrm{i}:=\sqrt{-1}$, $r:=|x|$ for $x\in\mathbb{R}^N$, and $\partial_r u:=\hat x\cdot\nabla_x u$ with $\hat x:=x/|x|\in\mathbb{S}^{N-1}$. The limit in \eqref{eq:Helm1} is called the radiation condition and characterizes the outgoing nature of the scattered wave field $u^s$. In the physical setup, when $N=2$, \eqref{eq:Helm1} describes the transverse electromagnetic scattering, where $\sigma^{-1}$ and $\mathbf{n}^2$ specify the electric permittivity and magnetic permeability of the optical medium \cite{LL}; whereas when $N=3$, \eqref{eq:Helm1} describes the acoustic scattering, where $\sigma^{-1}$ and $\mathbf{n}^2$ are respectively the density and modulus of the acoustic medium \cite{LSSZ}. We refer to \cite{LSSZ, Mclean2000} for the well-posedness of the scattering problem \eqref{eq:Helm1} with a unique solution $u\in H_{loc}^1(\mathbb{R}^N)$. It holds that:
\begin{equation}\label{eq:farfield}
u(x)=u^i(x)+\frac{e^{\mathrm{i}kr}}{r^{(N-1)/2}} u_\infty(\hat x)+\mathcal{O}\left(\frac{1}{r^{(N+1)/2}}\right)\quad\mbox{as}\ \ r\rightarrow+\infty. 
\end{equation}
In \eqref{eq:farfield}, $u_\infty(\hat x)$ is known as the far-field pattern, which encodes the scattering information caused by the perturbation of the incident field $u^i$ due to the scatterer $(\Omega;\sigma, \mathbf{n})$.

Associated with the scattering problem describe above, a practical inverse scattering problem of industrial importance is to recover $(\Omega;\sigma, \mathbf{n})$ by knowledge of $u_\infty(\hat x)$. The inverse problem can be abstractly recast as the following operator equation: 
\begin{equation}\label{eq:ip1}
\mathcal{F}(\Omega; \sigma, \mathbf{n})=u_\infty(\hat x),
\end{equation}
where $\mathcal{F}$ is defined via the scattering problem \eqref{eq:Helm1}. For \eqref{eq:ip1}, one peculiar case is that $u_\infty\equiv 0$. In such a case, the scatterer $(\Omega;\sigma,\mathbf{n})$ produces no scattering information to the outside observation, namely it is invisible/transparent with respect to the wave probing. Noting that $u^i\equiv 0$ readily yields $u^i=u^s$ in $\mathbb{R}^N\backslash\overline{\Omega}$ by Rellich's Theorem \cite{CK}, one can directly derive that $u|_{\Omega}\in H^1(\Omega)$ and $v:=u^i|_{\Omega}\in H^1(\Omega)$ fulfil that:
\begin{equation}\label{eq:trans0}
\left\{
\begin{array}{ll}
\nabla \cdot (\sigma \nabla u)+k^2 \mathbf{n}^2 u =0  & \text{in} \  \ \Omega, \medskip \\
\Delta v + k^2 v =0  & \text{in} \ \ \Omega, \medskip \\
\displaystyle u=v, \  \sigma\frac{\partial u}{\partial \nu}=\frac{\partial v}{\partial \nu} & \text{on} \ \ \partial\Omega,\medskip
\end{array}
\right.
\end{equation}
where $\nu\in\mathbb{S}^{N-1}$ is the exterior unit normal vector to $\partial\Omega$. \eqref{eq:trans0} is referred to as the transmission eigenvalue problem. It is clear that $u=v\equiv 0$ are trivial solutions to \eqref{eq:trans0}. If there exists a nontrivial pair of solutions $(u, v)$, $k$ is called a the transmission eigenvalue and $u, v$ are the associated transmission eigenfunctions. Clearly, according to our discussion above, the transmission eigenfunctions depict the wave propagation inside the scatterer when invisibility/transparency occurs.

The spectral theory of transmission eigenvalue problem has received considerable interest in the literature. We refer to \cite{CCH,CKreview,Liureview} for survey and review on the spectral properties of transmission eigenvalues. Recently, several intrinsic local and global geometric patterns of the transmission eigenfunctions have been revealed. Roughly and heuristically speaking, the transmission eigenfunctions tend to (globally) localize/concentrate on $\partial\Omega$ while (locally) vanish around singular/high-curvature points on $\partial\Omega$. Here, by localization/concentration, we mean that the $L^2$-energies of the eigenfunctions in $\Omega$ are localized/concentrated around $\partial\Omega$; and by a singular point, we mean the boundary point on $\partial\Omega$ at which the normal vector $\nu$ is no longer differentiable. A singular point can be regarded as having (extrinsic) curvature being infinity. The local geometric property was first discovered and investigated in \cite{BL1,BL2}, and was further studied in \cite{B,BLLW,BLin,BLX2020,DDL,DCL,DLS,DLWY} for different geometric and physical setups. We also refer to \cite{B,BLin,BL3,BL4,BLX2020,BPS,CV,CX,CDL,DCL,DLS,LX,SPV,SS} for related studies in characterizing non-scattering waves (locally) around corner/singular points. The global geometric property was first discovered and investigated in \cite{CDHLW}, and was further studied in \cite{CDLS,DJLZ21,DLWW} for different geometric and physical setups. Those geometric patterns are physically interpretable. In fact, in order to achieve invisibility/transparency, the wave propagates in a ``smart" way which slides over the boundary surface of the scattering object while avoids the singular/highly-curved places to avoid being trapped. More intriguingly, the geometric properties have been used to produce super-resolution imaging schemes for inverse acoustic and electromagnetic scattering problems \cite{CDHLW,HLW}, artificial mirage \cite{DLWW}, and pseudo surface plasmon resonance \cite{DLWW}. 

In this paper, we further study the boundary concentration of the transmission eigenfunctions and extend the related studies in \cite{CDHLW,CDLS} to more general setups. Specifically, in \cite{CDHLW,CDLS}, the theoretical justifications are mainly concerned with the case that $\sigma\equiv 1$, $\mathbf{n}$ is constant and $\Omega$ is smooth and convex, though the general case with variable medium parameter $\mathbf{n}$ and non-convex and non-smooth $\Omega$ is numerically investigated. We shall include both (being possibly variable) $\sigma$ and $\mathbf{n}$ into the current study and rigorously justify the boundary concentration phenomenon for the transmission eigenfunctions. Moreover, we shall present some novel numerical observations which strengthen the medium effect on the geometric patterns of the transmission eigenfunctions. More detailed discussion about the main results shall be given in Section 2, and the corresponding proofs are provided in Sections 3 and 4. 

\section{Statement of the main results and discussion}

Following \cite{CDHLW}, we first provide a quantitative description of surface/boundary localization of a function $\varphi\in L^2(\Omega)$. In what follows, for $\varepsilon\in\mathbb{R}_+$, we define
\begin{equation}\label{eq:dd1}
\mathcal{N}_\varepsilon(\partial\Omega):=\{x\in\Omega; \ \mathrm{dist}(x, \partial\Omega)<\varepsilon\},
\end{equation}
where $\mathrm{dist}$ signifies the Euclidean distance in $\mathbb{R}^N$, $N=2,3$. Clearly, $\mathcal{N}_\varepsilon(\partial\Omega)$ defines an $\varepsilon$-neighbourhood of $\partial\Omega$.

\begin{definition}\label{def:1}
A function $\varphi\in L^2(\Omega)$ is said to be \emph{boundary-localized} (or, \emph{surface-localized}) if there exists $\varepsilon\ll 1$ such that
\begin{equation}\label{eq:localized1}
\frac{\|\varphi\|_{L^2(\Omega\backslash\mathcal{N}_\varepsilon(\partial\Omega))}}{\|\varphi\|_{L^2(\Omega)}}\ll 1.
\end{equation}
\end{definition}
According to Definition~\ref{def:1}, the $L^2(\Omega)$-energy of the function $\varphi$ is mainly localized in a small neighbourhood of $\partial\Omega$, namely $\mathcal{N}_\varepsilon(\partial\Omega)$. In what follows, the asymptotic parameters involved in Definition~\ref{def:1} shall become more rigorous. In the case that $\Omega$ is a ball in $\mathbb{R}^N$, $N=2,3$, by scaling and translation if necessary, we can assume without loss of generality that $\Omega$ is the unit ball, namely $\Omega:=\{x\in\mathbb{R}^N; |x|<1\}$. In such a case, we also set $\Omega_\tau:=\{x\in\mathbb{R}^N; |x|<\tau\},\ \tau\in (0, 1)$  to signify the ball of radius $\tau$. 

Our main results can be stated as follows.
\begin{theorem}\label{thm:main1}
Consider the transmission eigenvalue problem \eqref{eq:trans1}. Let $\Omega$ be the unit ball in $\mathbb{R}^N$, $N=2,3$. Assume that $\mathbf{n}$, $\sigma$ are functions of the radial parameter $r$ only which fulfil Assumption A in what follows. Then for any given $\tau\in (0, 1)$, there exists a sequence of eigenfunctions $\{u_m, v_m\}_{m\in\mathbb{N}}$ associated to eigenvalues $k_m\rightarrow\infty$ as $m\rightarrow\infty$ such that:
\begin{equation}\label{eq:result1}
\lim_{m\rightarrow\infty} \frac{\|v_m\|_{L^2(\Omega_\tau)}}{\|v_m\|_{L^2(\Omega)}}=0.
\end{equation}
\end{theorem}

According to Theorem~\ref{thm:main1}, if one takes $\tau$ to be sufficiently close to 1, namely $\varepsilon=1-\tau\ll 1$, it is clear that  $v_m$ in \eqref{eq:result1} with $m$ sufficiently large are all boundary-localized according to Definition~\ref{def:1}. The Assumption A in Theorem~\ref{thm:main1} is stated as follows.

\medskip
{\bf{Assumption A.}} Let $\mathbf{n}_{1}$, $\mathbf{n}_{2}$,
$M_1$, $M_2$ be positive constants, and the radial functions
$\mathbf{n}(r)\in C[0, 1]$, $\sigma(r)\in C^2[0, 1]$ satisfy the following properties:
\begin{equation}\label{eq:assump1}
\begin{array}{ll}
(A1)~1<\mathbf{n}_{1}<\mathbf{n}(r)<\mathbf{n}_{2}, \quad &(A2)~1<M_1<M_2,   \\
(A3)~\sigma'(r)\leq0, \quad \quad \quad\quad   &(A4)~M_1\leq
\frac{\mathbf{n}^2}{\sigma(r)}\leq
M_2,\\
(A5)\sigma'(r)^2\geq2\sigma''(r)\sigma(r).
\end{array}
\end{equation}
It is remarked that in our subsequent analysis, we actually can combine conditions (A4) and (A5) by requiring a slightly less restrictive condition:
\begin{equation}\label{eq:cc1}
M_1 k^2\leq \frac{\mathbf{n}^2}{\sigma} k^2+\frac{\sigma'^2}{4\sigma^2(r)}-\frac{\sigma''(r)}{2\sigma(r)}\leq M_2 k^2. 
\end{equation}
However, the condition \eqref{eq:cc1} involves the eigenvalue $k^2$, and we split it into the two conditions (A4) and (A5) in Assumption A. It is directly verified that when both $\sigma$ and $\mathbf{n}$ are constant, the assumptions in \eqref{eq:assump1} yield that
\begin{equation}\label{eq:assump2}
\mathbf{n}>1\quad\mbox{and}\quad 0<\sigma\leq \mathbf{n}^2.
\end{equation}
Nevertheless, if both $\sigma$ and $\mathbf{n}$ are constant, we can prove a stronger boundary-localization result.

\begin{theorem}\label{thm:main12}
Consider the same setup as Theorem~\ref{thm:main1} and assume that $\mathbf{n}$, $\sigma$ are both positive constants satisfying $\mathbf{n}^2>\sigma$. Then for any given $\tau\in (0, 1)$, there exists a sequence of eigenfunctions $\{u_m, v_m\}_{m\in\mathbb{N}}$ associated to eigenvalues $k_m\rightarrow\infty$ as $m\rightarrow\infty$ such that
\begin{equation}\label{eq:result1}
\lim_{m\rightarrow\infty} \frac{\|\varphi_m\|_{L^2(\Omega_\tau)}}{\|\varphi_m\|_{L^2(\Omega)}}=0,\quad\varphi_m=u_m\ \mbox{or}\ v_m. 
\end{equation}

\end{theorem}

That is, if $\sigma$ and $\mathbf{n}$ are both constant, there exists a sequence of transmission eigenfunctions where both the $u$-parts and $v$-parts are localized around $\partial \Omega$. 

\begin{remark}\label{rem:n1}
In Theorem~\ref{thm:main12}, the condition $\mathbf{n}^2>\sigma$ is required. We would like to point that the condition can also be replaced to be $\mathbf{n}^2<\sigma$. In fact, in the latter case, for $(u, v)$ to \eqref{eq:trans0}, we can set 
\begin{equation}\label{eq:rep1}
\tilde k=k\cdot\frac{\mathbf{n}}{\sigma},\ \ \tilde{\mathbf{n}}=\frac{1}{\mathbf{n}},\ \ \tilde{\sigma}=\frac{1}{\sigma},\ \ \tilde{u}=v,\ \ \tilde{v}=u. 
\end{equation}
It is directly verified that 
\begin{equation}\label{eq:trans0n1}
\left\{
\begin{array}{ll}
\nabla \cdot (\tilde\sigma \nabla \tilde u)+\tilde{k}^2 \tilde{\mathbf{n}}^2 \tilde{u} =0  & \text{in} \ \  \Omega, \medskip \\
\Delta \tilde v + \tilde k^2 \tilde v =0  & \text{in} \ \ \Omega, \medskip \\
\displaystyle \tilde u=\tilde v, \  \tilde \sigma\frac{\partial \tilde u}{\partial \nu}=\frac{\partial \tilde v}{\partial \nu} & \text{on} \ \ \partial\Omega. \medskip \\
\end{array}
\right.
\end{equation}
Since $\mathbf{n}^2<\sigma$, one clearly has $\tilde{\mathbf{n}}^2>\tilde\sigma$. Hence, Theorem~\ref{thm:main12} applied to \eqref{eq:trans0n1} readily yields the existence of a sequence of boundary-localized transmission eigenfunctions $\{\tilde{u}_m, \tilde{v}_m\}_{m\in\mathbb{N}}$ associated with $\tilde{k}_m\rightarrow\infty$ as $m\rightarrow\infty$. Then by using the relations in \eqref{eq:rep1}, we have the existence of a sequence of boundary-localized transmission eigenfunctions $\{{u}_m, {v}_m\}_{m\in\mathbb{N}}$ associated with ${k}_m\rightarrow\infty$ as $m\rightarrow\infty$. Finally, we would like to point out that if one takes $\sigma\equiv 1$, the result in Theorem~\ref{thm:main12} recovers those in \cite{CDHLW, DJLZ21}. 
\end{remark}

So far, we have mainly considered the radially symmetric cases. In particular, in Theorem~\ref{thm:main1}, we can only show the boundary-localization of the $v_m$-part, though we believe that there exist infinitely many transmission eigenfunctions such that both $u$- and $v$-parts are boundary-localized. In \cite{CDHLW}, extensive numerical examples show that the surface/boundary-localization is a generic phenomenon occurring for transmission eigenfunctions, even associated with variable medium parameters and general domains. In particular, we note that for the case that $\Omega$ is divided into two connected subdomains, $\Omega=\Omega^{(1)}\cup\Omega^{(2)}$ with $\Omega^{(1)}\Subset\Omega$ and $\Omega^{(2)}=\Omega\backslash\overline{\Omega^{(1)}}$, if the medium parameters $(\sigma_1, \mathbf{n}_1)$ in $\Omega^{(1)}$ and $(\sigma_2, \mathbf{n}_2)$ in $\Omega^{(2)}$ are both constant, then there exist boundary-localized transmission eigenmodes generically even when $(\sigma_1, \mathbf{n}_1)\neq (\sigma_2, \mathbf{n}_2)$ (which corresponds to variable medium parameters). It is remarked that in the numerical examples in \cite{CDHLW}, it is always assumed that $\sigma_1=\sigma_2\equiv 1$. Nevertheless, we would like to confirm that the same numerical conclusion holds   
when $\sigma_1$ and $\sigma_2$ are other constants; see Fig.~\ref{fig:1} for typical illustration and comparison. It can be seen that the boundary-localization phenomenon is still every evident though less sharper than the case with $\sigma\equiv 1$. 
\begin{figure}[htp]
  \centering
  \includegraphics[width=0.7\textwidth]{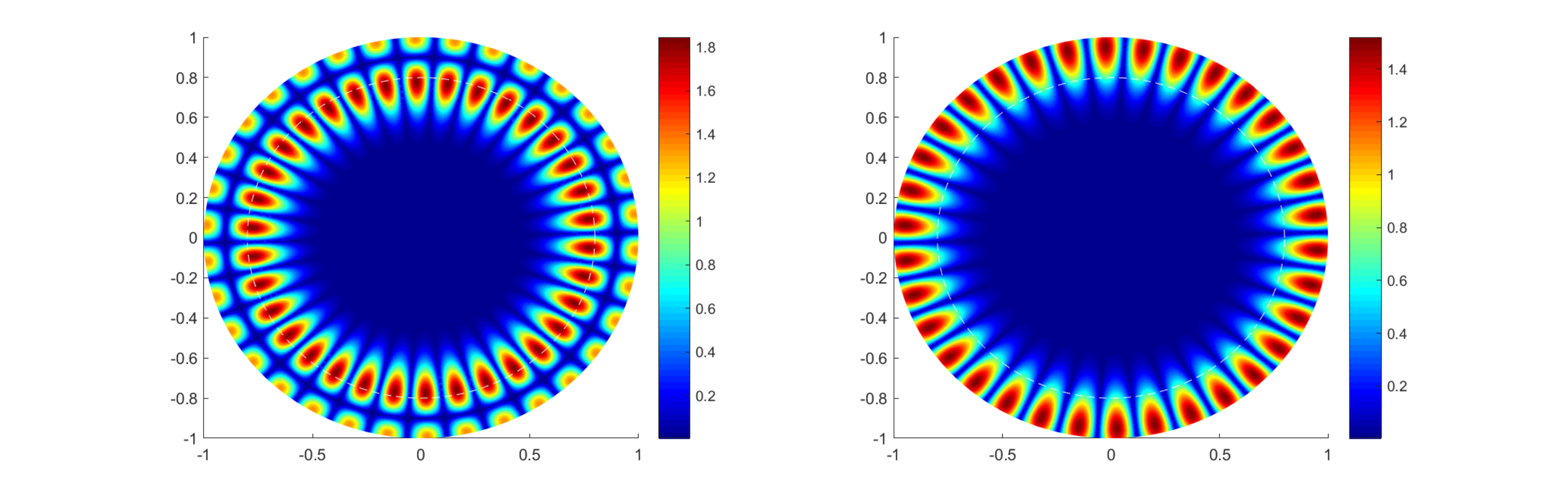} \includegraphics[width=0.7\textwidth]{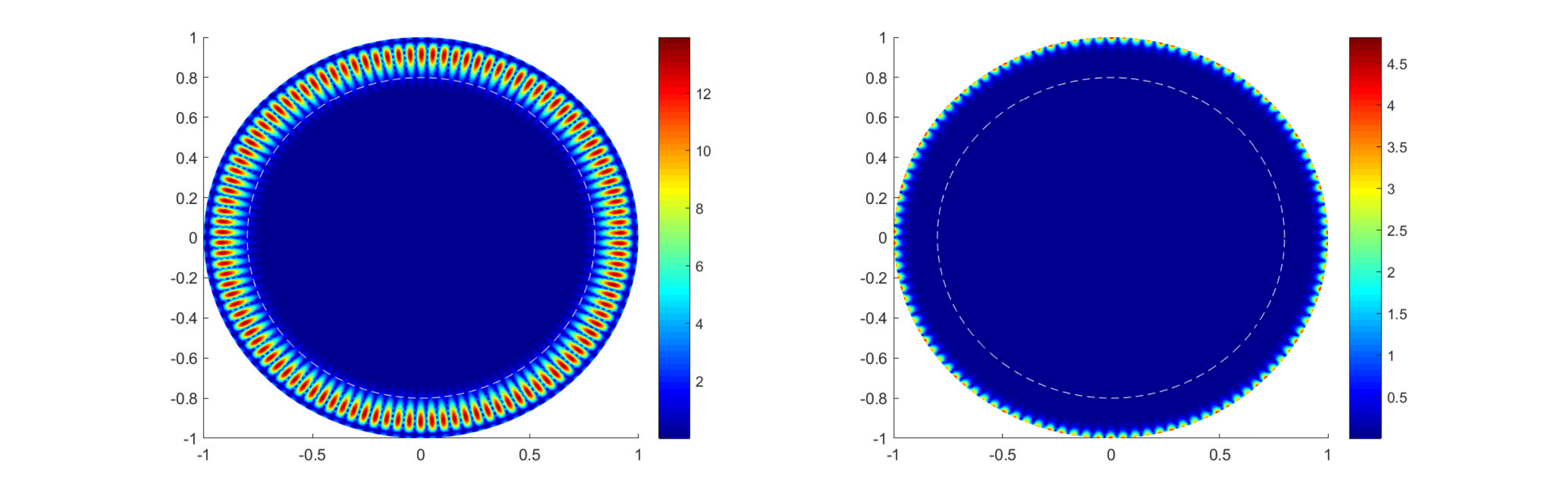}
  \caption{First row: $|u|$ (left) and $|v|$ (right) with $k^2=402.989$ and $(\sigma, \mathbf{n}^2)=(1, 1)\chi_{\Omega_{0.8}}+(2, 3)\chi_{\Omega_1\backslash\Omega_{0.8}}$. Second row: $|u|$ (left) and $|v|$ (right) with $k^2=402.665$ and $(\sigma, \mathbf{n}^2)=(1, 0.1)\chi_{\Omega_{0.8}}+(1, 10)\chi_{\Omega_1\backslash\Omega_{0.8}}$. \label{fig:1} }
\end{figure}
Next, we present two more numerical examples which were not considered in \cite{CDHLW}, and show that variable medium parameters can make the geometric patterns of the transmission eigenfunctions more intriguing; see Fig.~\ref{fig:2}, where $\Omega^{(1)}$ and $\Omega^{(2)}$ respectively signify the left-subdomain and right-subdomain of the domain $\Omega$. 
\begin{figure}[htp]
  \centering
  \includegraphics[width=0.7\textwidth]{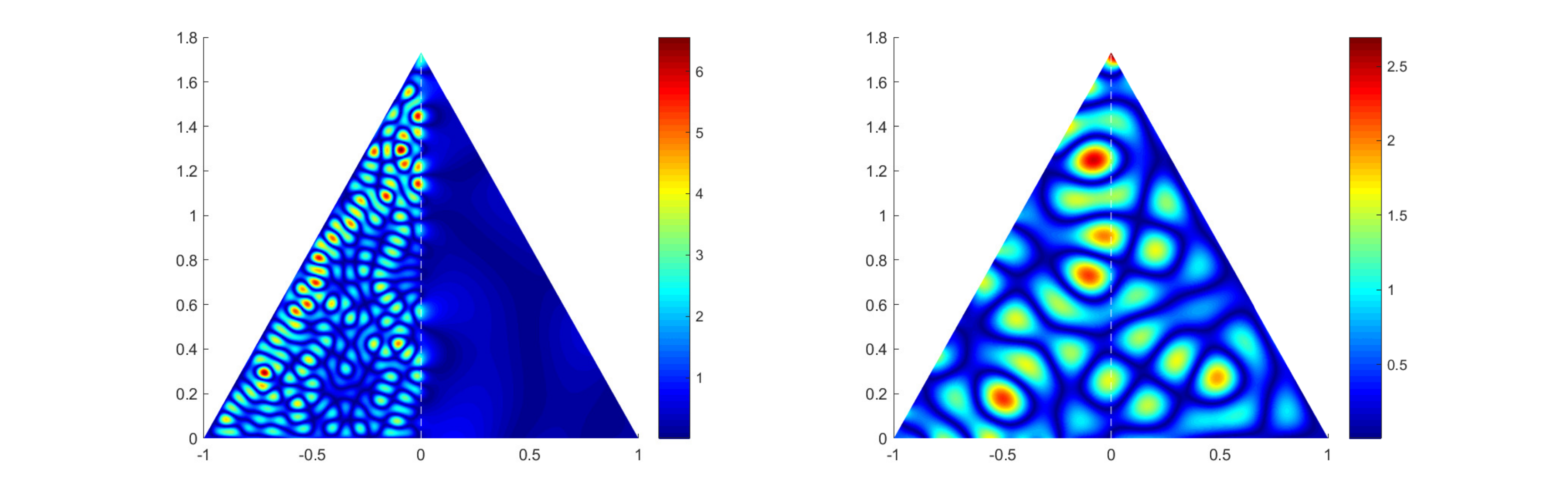}   \includegraphics[width=0.7\textwidth]{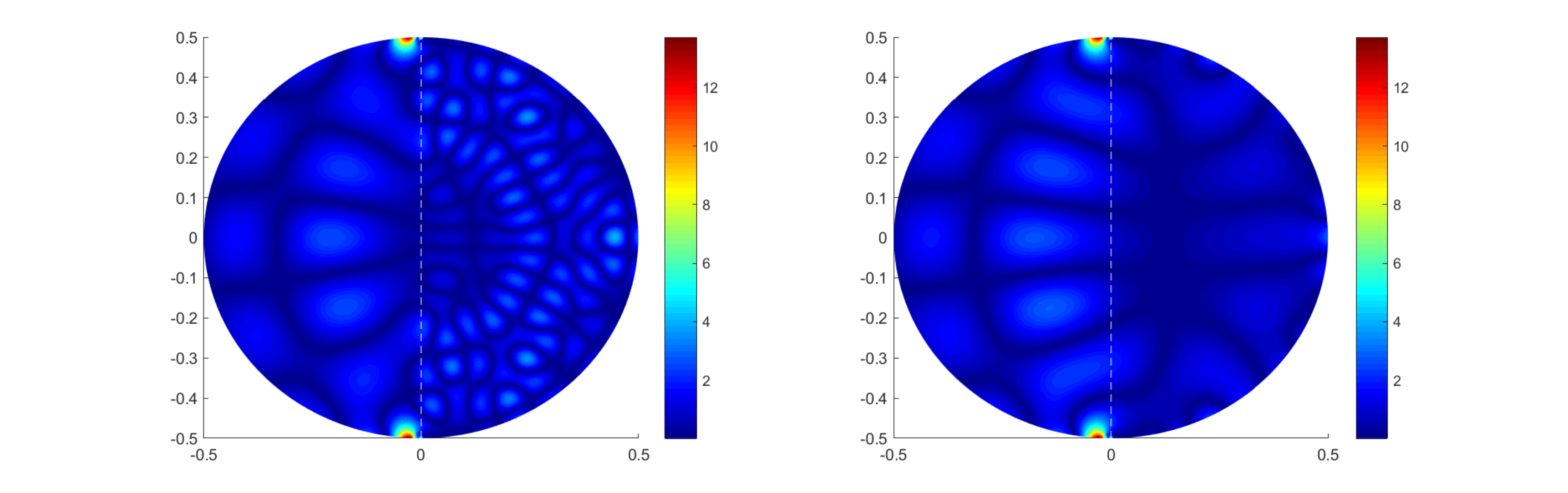}
  \caption{First row: $|u|$ (left) and $|v|$ (right) with $k^2=395.307$ and $(\sigma, \mathbf{n}^2)=(0.1, 1)\chi_{\Omega^{(1)}}+(10, 1)\chi_{\Omega^{(2)}}$. Second row: $|u|$ (left) and $|v|$ (right) with $k^2=410.024$ and $(\sigma, \mathbf{n}^2)=(1, 1.1)\chi_{\Omega^{(1)}}+(1, 10)\chi_{\Omega^{(2)}}$. \label{fig:2}}
\end{figure}
It can be observed that in addition to the boundary localization, it may happen that the transmission eigenfunctions are localized around the material interface or even at two ``exceptional" points. The numerical observations in Fig.~\ref{fig:2} partly corroborates the necessity of introducing Assumption A in our current study. On the other hand, they are highly interesting spectral phenomena that are worth further investigation in our forthcoming work.

\section{Proof of Theorem~\ref{thm:main1}}

%

In this section, we present the proof of Theorem~\ref{thm:main1}. That is, the transmission eigenvalue problem is given by:
\begin{equation}\label{eq:trans1}
\left\{
\begin{array}{ll}
\nabla \cdot (\sigma(r) \nabla u)+k^2 \mathbf{n}^2(r)u =0  & \text{in} \ \ \Omega, \medskip \\
\Delta v + k^2 v =0  & \text{in} \ \  \Omega, \medskip \\
\displaystyle u=v, \  \sigma(1)\frac{\partial u}{\partial r}=\frac{\partial v}{\partial r} & \text{on} \ \ \partial\Omega. \medskip \\
\end{array}
\right.
\end{equation}
Throughout the rest of this section, we assume that $\Omega$ is the unit ball in $\mathbb{R}^N$, $N=2, 3$. We shall divide our analysis into two parts, respectively, for the two and three dimensions.

\subsection{Two-dimensional case}

In two dimensions, we let $x=(r\cos\theta, r\sin\theta)\in\mathbb{R}^2$, $r=|x|$, denote the polar coordinate. In the sequel, $J_m$, $m\in\mathbb{N}\cup\{0\}$, signifies the $m$-th order Bessel function \cite{CK,LZ}. 

First, we know that the solutions to \eqref{eq:trans1} have the following Fourier expansions \cite{CK,LZ}:
\begin{equation}\label{eq:series1}
u(x)=\sum_{m=0}^\infty \alpha_m \phi_m(r;k) e^{\mathrm{i}m\theta}, \quad v(x)=\sum_{m=0}^\infty \beta_m J_m(kr) e^{\mathrm{i}m\theta},
\end{equation}
where $\alpha_m$, $\beta_m\in\mathbb{C}$. Set
\begin{equation}\label{eq:series2}
u_m(x)=\alpha_m \phi_m(r;k) e^{\mathrm{i}m\theta}, \quad v_m(x)=\beta_m J_m(kr) e^{\mathrm{i}m\theta},
\end{equation}
then we have
\begin{equation}
\sigma\phi_m^{\prime\prime}+\left(\sigma^{\prime}+\frac{\sigma}{r}\right)\phi_m^{\prime}+\left(
k^2\mathbf{n}^2-\frac{m^2}{r^2}\sigma\right)\phi_m=0,
\end{equation}
where the differentiations are with respect to the variable $r$. Furthermore, we assume that
\begin{equation*}
\psi_m(r;k)=(r\sigma(r))^{\frac{1}{2}}\phi_m(r;k),
\end{equation*}
and then
\begin{equation*}
\psi_m^{''}+\left\{\frac{k^2\mathbf{n}^2}{\sigma}-\frac{m^2}{r^2}+\frac{\sigma^2+r^2\sigma^{\prime 2}+2r\sigma\sigma^{\prime}}{4r^2\sigma^2}-\frac{2\sigma^{\prime}+r\sigma^{\prime\prime}}{2r\sigma}\right\}\psi_m=0,
\end{equation*}
which implies
\begin{equation}\label{eq:def fm}
\psi_m^{''}+\left\{\left(\frac{k^2\mathbf{n}^2}{\sigma}+\frac{\sigma^{\prime2}}{4\sigma^2}-\frac{\sigma^{\prime\prime}}{2\sigma}\right)
-\frac{\sigma^{\prime}}{2\sigma r}-\left(m^2-\frac{1}{4}\right)\frac{1}{r^2}\right\}\psi_m=0.
\end{equation}

For the subsequent use, we let $j_{m,s}$ and $j'_{m,s}$ denote the $s$-th positive root of
$J_m(t)$ and $J'_m(t)$, respectively, that are arranged according to the magnitudes \cite{LZ}.



\begin{lemma}\label{lem:2}
Under Assumption A, for any $m \geq 1$ the function $\psi_m(r;k)$ in \eqref{eq:def fm} possesses at least one zero point in $\left(\frac{j_{m,s_0}}{k\sqrt{M_1}}, \frac{j_{m, s_0+1}}{k\sqrt{M_1}}\right)$ for any $s_0 \in \mathbb{N}$.
\end{lemma}

\begin{proof}
The condition (A1)
guarantees the non-degeneracy of $u$ in \eqref{eq:trans1}. Let
\begin{equation}
\left\{
\begin{array}{lll}
P_{m}(r;k)&=&M_1k^2-(m^2-\frac{1}{4})\frac{1}{r^2},\\
Q_m(r;k)&=&\left(\frac{k^2\mathbf{n}^2}{\sigma}+\frac{\sigma^{\prime2}}{4\sigma^2}-\frac{\sigma^{\prime\prime}}{2\sigma}\right)
-\frac{\sigma^{\prime}}{2\sigma r}+\left(m^2-\frac{1}{4}\right)\frac{1}{r^2},
\end{array}
\right.
\end{equation}
and it follows from (A2)-(A5) that
\begin{equation}\label{eq:Pm Qm}
P_{m}(r;k)\leq Q_m(r;k).
\end{equation}
Consider the following differential equations
\begin{eqnarray}
(\mathbf{P}1)~~ \qquad \xi_m^{\prime\prime}+P_{m}(r;k)\xi_m=0,   \\
(\mathbf{P}2) \qquad \psi_m^{\prime\prime}+Q_m(r;k)\psi_m=0. 
\end{eqnarray}
Then for each fixed $m\in\mathbb{N}$ , the solutions of the $(\mathbf{P}1)$ are denoted by
$J_m(k\sqrt{M_1}r)$, whose roots are given by $\{\frac{j_{m,s_0}}{k\sqrt{M_1}}\}_{s_0 \in \mathbb{N}}$.
It follows from the Singular Strum theorem in \cite{DU2010} and \eqref{eq:Pm Qm} that the solution of
$(\mathbf{P}2)$ has at least one zero in
$\left(\frac{j_{m,s_0}}{k\sqrt{M_1}},\frac{j_{m,s_0+1}}{k\sqrt{M_1}}\right)$ for any $s_0 \in \mathbb{N}$, denoted by $t_{m,s_{0}}(k)$; see the dashed line in Fig~\ref{fig:tmso}.
\end{proof}

\begin{figure}[htp]
\centering
\includegraphics{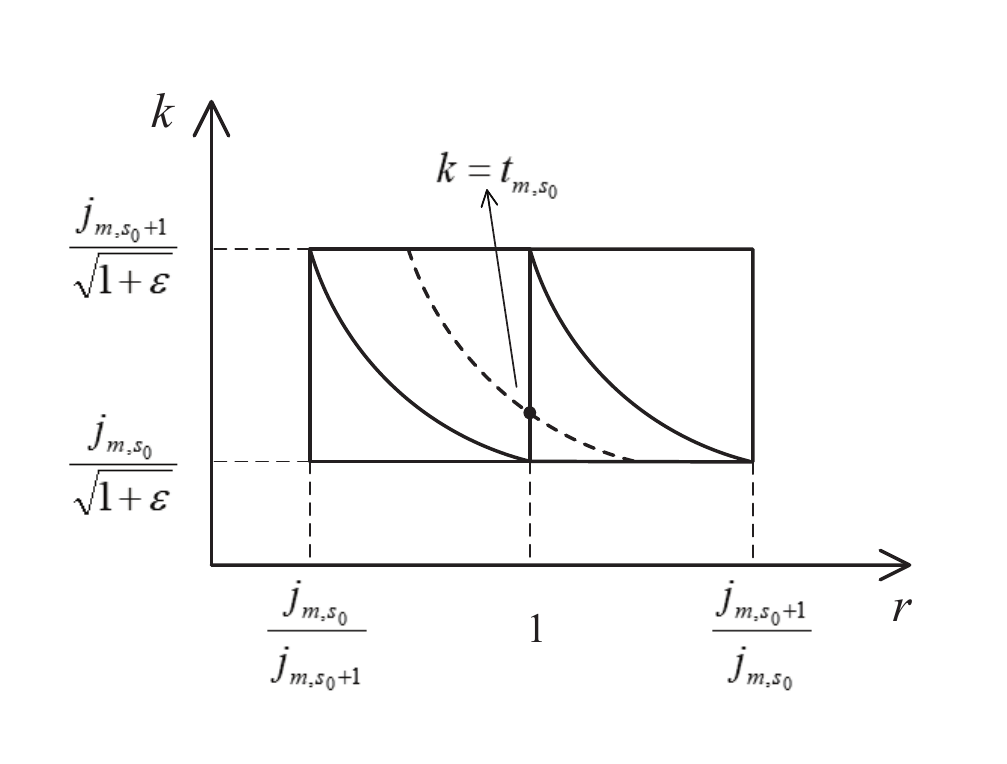}
\caption{Schematic demonstration of the curve $t_{m,s_{0}}(k)$ and the point $z_{m,s_0}$.}\label{fig:tmso}
\end{figure}

Since the solutions of $(\mathbf{P}1)$ and $(\mathbf{P}2)$ are continuous with respect to the parameter
$k$, there exists a number $z_{m,s_0}\in
\left(\frac{j_{m,s_0}}{\sqrt{M_1}},\frac{j_{m,s_0+1}}{\sqrt{M_1}}\right)$
such that
\begin{equation}
\psi_m(1;z_{m,s_0})=0,
\end{equation}
which is the intersection of the dashed line $t_{m,s_{0}}(k)$ and $r=1$ in Fig~\ref{fig:tmso}.

Let $\beta_m=1$ in \eqref{eq:series2}. Using the transmission condition, it holds that
\begin{equation*}
\alpha_m=\frac{J_{m}(k)}{\phi_{m}(1;k)}.
\end{equation*}
Set $f_m(k)=\sigma(1)\frac{\partial \phi_m(1;k)}{\partial r}J_m(k)-kJ^{\prime}_m(k)\phi_m(1;k)$. Using the recursive formula of Bessel functions \cite{CK}, we have
\begin{equation}\label{eq:fmzs0fmzs1}
f_m(k)=\sigma(1)
\frac{\partial \phi_m(1;k)}{\partial r}J_m(k)+(mJ_{m}(k)-kJ_{m-1}(k))\phi_m(1;k),\quad m\geq 1.
\end{equation}

Next, we find the roots of $f_m(k)$ on the interval $\left(\frac{j_{m,s_0}}{\sqrt{M_1}},\frac{j_{m,s_0+2}}{\sqrt{M_1}}\right)$.
\begin{lemma}\label{lem:3.2}
Under Assumption A, for any given $s_0 \in \mathbb{N}$, there exists $m_0(\mathbf{n},\sigma)\in \mathbb{N}$, depending on $\mathbf{n}$ and $\sigma$, such that when $m>m_0(\mathbf{n},\sigma)$, the function $f_m(k)$ in \eqref{eq:fmzs0fmzs1} possesses at least one root in $\left(\frac{j_{m,s_0}}{\sqrt{M_1}}, \frac{j_{m, s_0+2}}{\sqrt{M_1}}\right)$.
\end{lemma}

\begin{proof}
In the following, we let $a_s$ denote the $s$-th negative zero of the Airy function \cite{WQ99}:
\begin{equation*}\label{eq:dd1}
a_{s}=-\left[\frac{3 \pi}{8}(4 s-1)\right]^{2 /
3}(1+\sigma_{s}),
\end{equation*}
where $\sigma_s$ satisfies the following estimate:
\begin{equation*}
0 \leq \sigma_{s} \leq 0.130\left[\frac{3 \pi}{8}(4
s-1.051)\right]^{-2}.
\end{equation*}
By (1.2) in \cite{WQ99}, it holds that
\begin{equation}\label{eq:jms bdd}
m-\frac{a_{s}}{2^{1 / 3}} m^{1 / 3}<j_{m, s}<m-\frac{a_{s}}{2^{1 /
3}} m^{1 / 3}+\frac{3}{20} a_{s}^{2} \frac{2^{1 / 3}}{m^{1 / 3}},
\end{equation}
where we recall that $j_{m,s}$ denotes the $s$-th positive root of $J_m(t)$. For each fixed $s_0 \in \mathbb{N}$, when $m(s_0)$ is large enough, we have
\begin{equation*}
\frac{j_{m,s_0}}{\sqrt{M_1}}<j_{m,1}^{\prime}.
\end{equation*}
Consider the interval $\left(\frac{j_{m,s_0}}{\sqrt{M_1}},\frac{j_{m,s_0+2}}{\sqrt{M_1}}\right)$, there exist at least two consecutive zeros $z_{m,s_{0}}$ and $z_{m,s_{0}+1}$ of
$\psi_m(1;k)$, as well as $\phi_m(1;k)$. This result together with the monotonicity of $J_{m}$ in the interval $[0,j_{m,1}^{'}]$ readily yields that
\begin{equation}\label{eq:Jzms0Jzms01}
\begin{array}{ll}
J_{m}(z_{m,s_{0}})J_{m}(z_{m,s_{0}+1})< J_m^2(j_{m,1}^{'}).
\end{array}
\end{equation}

By virtue of \eqref{eq:Jzms0Jzms01} and \eqref{eq:fmzs0fmzs1}, it can be directly verified that
\begin{equation}
\begin{split}
& f_m(z_{m,s_{0}})f_m(z_{m,s_{0}+1})\\
=& \sigma^2(1)J_{m}(z_{m,s_{0}})J_{m}(z_{m,s_{0}+1})
\frac{\partial \phi_m(1;z_{m,s_{0}})}{\partial r}\frac{\partial \phi_m(1;z_{m,s_{0}+1})}{\partial r}\\
<& \sigma^2(1)J_m^2(j_{m,1}^{\prime})\frac{\partial \phi_m(1;z_{m,s_{0}})}{\partial r}\frac{\partial \phi_m(1;z_{m,s_{0}+1})}{\partial r}.
\end{split}
\end{equation}
Since $z_{m,s_{0}}$, $z_{m,s_{0}+1}$ are two consecutive roots, $\frac{\partial \phi_m(1;z_{m,s_{0}})}{\partial r}$ and $\frac{\partial \phi_m(1;z_{m,s_{0}+1})}{\partial r}$ have opposite signs. Hence, we have
\begin{equation}\label{eq:rr1}
f_m(z_{m,s_{0}})f_m(z_{m,s_{0}+1})<0.
\end{equation}
By applying Rolle's theorem to \eqref{eq:rr1}, we immediately see that there exists at least one zero point of $f_m(k)$ in $(z_{m,s_{0}},z_{m,s_{0}+1})$.

The proof is complete.
\end{proof}

Lemma~\ref{lem:3.2} shows the existence of transmission eigenvalues to \eqref{eq:trans1} associated with spherically stratified media. It is noted that in \cite{LC}, transmission eigenvalues were calculated in a similar setup but assuming $\sigma\equiv 1$. In what follows, for a fixed $s_{0}$, we let the transmission eigenvalue be denoted by
\begin{equation}\label{eq:te1}
k_{l_{m}} :=k_{m,s_{0}}\in(z_{m,s_{0}},z_{m,s_{0}+1}),\quad m=m_0+1,m_0+2,... ,
\end{equation}
where $m_0 = m_0(\mathbf{n},\sigma)$ is sufficiently large.
\begin{theorem}\label{thm:surloc}
Consider the same setup as Theorem~\ref{thm:main1} in $\mathbb{R}^2$ and assume that Assumption A holds. Let $(u_m, v_m)$ be the pair of eigenfunctions associated with $k_{l_m}$ in \eqref{eq:te1}. We have
\begin{equation}\label{eq:ff1}
\lim\limits_{m\rightarrow\infty}\frac{\|v_m\|_{L^2(\Omega_{\tau})}}{\|v_m\|_{L^2(\Omega)}}=0.
\end{equation}
\end{theorem}
\begin{proof}
Let $\beta_{m}=1$ in \eqref{eq:series1}. By direct calculations, we have
\begin{equation*}
\left\|v_{m}\right\|_{L^{2}(\Omega_{\tau})}^{2} =\int_{\Omega_{\tau}}\left|J_{m}(k_{l_{m}}|x|)\right|^{2} \mathrm{d} x =2 \pi \int_{0}^{\tau}rJ^{2}_{m}(k_{l_{m}} r)
\mathrm{d} r,
\end{equation*}
which in particular gives that
\begin{equation*}
\left\|v_{m}\right\|_{L^{2}(\Omega)}^{2} =2 \pi
\int_{0}^{1}rJ^{2}_{m}(k_{l_{m}} r)  \mathrm{d} r\
\text{.}
\end{equation*}
Combining
\begin{equation*}
\frac{j_{m,s_0}}{\sqrt{M_1}} \leq z_{m,s_{0}} \leq k_{l_{m}} \leq z_{m,s_{0}+1} \leq \frac{j_{m,s_0+2}}{\sqrt{M_1}} ,\quad m=m_0+1,m_0+2,...
\end{equation*}
and the estimate \eqref{eq:jms bdd}, we know that for any $s_0 \in \mathbb{N}$, there exists $m_1(M_1,\sigma)\in \mathbb{N}$ such that when $m>m_1(M_1, \sigma)$,
the following estimate holds
\begin{equation*}
\frac{1}{\sqrt{M_1}} \leq \frac{k_{l_{m}}}{m} \leq \frac{1+\sqrt{M_1}}{2\sqrt{M_1}}.
\end{equation*}

Similar to the arguments in the proof of Theorem\,2.6 in \cite{DJLZ21}, for any $\tau \in (0,1)$,  there exists $\delta(\tau, M_1, \sigma) > 0$ such that
\begin{equation*}
\begin{split}
\frac{\left\|v_{m}\right\|_{L^{2}(\Omega_{\tau})}^{2}}{\left\|v_{m}\right\|_{L^{2}(\Omega)}^{2}}
&\leq \displaystyle \frac{\int_{0}^{\tau}rJ^2_{m}(k_{l_{m}}r)\mathrm{d}r}
{\int_{0}^{1}rJ^2_{m}(k_{l_{m}}r)\mathrm{d}r}\leq \displaystyle \frac{\tau^2J^2_{m}(k_{l_{m}}\tau)}{\displaystyle  \frac{\frac{1}{2}J^3_{m}(k_{l_{m}})}{J_{m}(k_{l_{m}})+2k_{l_{m}}J'_{m}(k_{l_{m}})}}\\
&\leq \displaystyle
2\tau^2\left(\frac{J_{m}(k_{l_{m}}\tau)}{J_{m}(k_{l_{m}})}\right)^2\left(1+2k_{l_{m}}\frac{J'_{m}(k_{l_{m}})}{J_{m}(k_{l_{m}})}\right)\\
&\leq \displaystyle
36\left(\sqrt{M_1}-1\right)m^4\tau^2\left(\frac{J_{m}(k_{l_{m}}\tau)}{J_{m}(k_{l_{m}})}\right)^2\\
&\leq \displaystyle 144\frac{\sqrt{M_1}}{(\sqrt{M_1}-1)^2}m^4\tau^2\Big(1-\delta(\tau, M_1, \sigma)\Big)^{2m},
\end{split}
\end{equation*}
which readily implies \eqref{eq:ff1}. 

The proof is complete.
\end{proof}

\subsection{Three-dimensional result}
The proof of Theorem~\ref{thm:main1} in three dimensions follows a similar argument to that of the two-dimensional case in Theorem~\ref{thm:surloc} . In what follows, we only sketch the necessary modifications in what follows. 

In three dimensions, we let $x=(r\cos\theta_1,r\sin\theta_1\cos\theta_2,r\sin\theta_1\sin\theta_2)\in\mathbb{R}^3$ denote the polar coordinate, where $r=|x|$, $\theta_1\in [0, \pi]$ and $\theta_2\in [0, 2\pi]$. Let $Y_{m}^{l}$ be the spherical harmonic function of order $m$ and degree $l$, and $j_{m}(|x|):=\sqrt{\frac{\pi}{2r}} J_{m+1/2}(r)$ be the {spherical Bessel function} \cite{CK, LZ}. The solutions to \eqref{eq:trans1} in $\mathbb{R}^3$ have the following Fourier expansions:
\begin{equation}
\begin{aligned}
u(x)&=\sum_{m=0}^{\infty} \sum_{l=-m}^{m} \alpha_{m}^{l} \phi_m(r;k) Y_{m}^{l}(\hat{x}),\\
v(x)&=\sum_{m=0}^{\infty} \sum_{l=-m}^{m} \beta_{m}^{l}
j_{m}(k r) Y_{m}^{l}(\hat{x}),
\end{aligned}
\end{equation}
where $\alpha_{l}^{m}, \beta_{l}^{m}\in\mathbb{C}$.
Set
\begin{equation}\label{fourier3}
u_{m}(x)=\alpha_m\phi_m(r; k)Y_{m}^{l}(\hat{x}),\quad v_{m}(x)=\beta_m j_m(kr)Y_{m}^{l}(\hat{x}).
\end{equation}
Then we have
\begin{equation}
\sigma\phi_m^{\prime\prime}+\left(\sigma^{\prime}+\frac{2\sigma}{r}\right)\phi_m^{\prime}+\left(
k^2\mathbf{n}^2-\frac{m(m+1)}{r^2}\sigma\right)\phi_m=0.
\end{equation}
where the derivatives are associated to the variable $r$. Furthermore, we assume that
\begin{equation}\label{eq:def psim}
\psi_m(r;k)=r\sigma^{\frac{1}{2}}(r)\phi_m(r;k).
\end{equation}
Then
\begin{equation}
\psi_m^{\prime\prime}+\left(\frac{k^2\mathbf{n}^2}{\sigma}-\frac{m(m+1)}{r^2}+\frac{\sigma^{\prime2}}{4\sigma^2}-\frac{\sigma^{\prime}}{r\sigma}-\frac{\sigma^{\prime\prime}}{2\sigma}\right)\psi_m=0,
\end{equation}
which implies
\begin{equation}
\psi_m^{\prime\prime}+\left\{\left(\frac{k^2\mathbf{n}^2}{\sigma}+\frac{\sigma^{\prime2}}{4\sigma^2}-\frac{\sigma^{\prime\prime}}{2\sigma}\right)
-\frac{\sigma^{\prime}}{\sigma r}-m\left(m+1\right)\frac{1}{r^2}\right\}\psi_m=0.
\end{equation}

\begin{lemma}\label{lem:4}
Under Assumption A, for any $m \geq 1$ the function $\psi_m(r;k)$ in \eqref{eq:def psim} possesses at least one zero point in $\left(\frac{j_{m+\frac{1}{2},s_0}}{k\sqrt{M_1}}, \frac{j_{m+\frac{1}{2}, s_0+1}}{k\sqrt{M_1}}\right)$ for any $s_0 \in \mathbb{N}$.
\end{lemma}

\begin{proof}
The condition (A1)
guarantees the non-degeneracy of $u$ in equation \eqref{eq:trans1}. Let
\begin{equation}
\left\{
\begin{array}{lll}
P_{m}(r;k)&=&M_3k^2-m(m+1)\frac{1}{r^2},\\
Q_m(r;k)&=&\left(\frac{k^2\mathbf{n}^2}{\sigma}+\frac{\sigma^{\prime2}}{4\sigma^2}-\frac{\sigma^{\prime\prime}}{2\sigma}\right)
-\frac{\sigma^{\prime}}{\sigma r}+m(m+1)\frac{1}{r^2},
\end{array}
\right.
\end{equation}
and it follows from (A2)-(A5) that
\begin{equation}\label{eq:Pm Qm3}
P_{m}(r;k)\leq Q_m(r;k).
\end{equation}
Consider the following differential equations:
\begin{eqnarray}
(\mathbf{P}1)~~ \qquad \xi_m^{\prime\prime}+P_{m}(r;k)\xi_m=0,   \\
(\mathbf{P}2) \qquad \psi_m^{\prime\prime}+Q_m(r;k)\psi_m=0.
\end{eqnarray}
Then for each fixed $m\in\mathbb{N}$ , the solutions of the $(\mathbf{P}1)$ are denoted by
$\sqrt{r}J_{m+\frac{1}{2}}(k\sqrt{M_1}r)$, whose roots are given by $\{\frac{j_{{m+\frac{1}{2}},s_0}}{k\sqrt{M_1}}\}_{s_0 \in \mathbb{N}}$.
It follows from the Singular Strum theorem in \cite{DU2010} and \eqref{eq:Pm Qm} that the solutions of
$(\mathbf{P}2)$ has at least one zero in
$\left(\frac{j_{{m+\frac{1}{2}},s_0}}{k\sqrt{M_1}},\frac{j_{{m+\frac{1}{2}},s_0+1}}{k\sqrt{M_1}}\right)$ for any $s_0 \in \mathbb{N}$, denoted by $t_{{m+\frac{1}{2}},s_{0}}(k)$.
\end{proof}

Since the solutions of $(\mathbf{P}1)$ and $(\mathbf{P}2)$ are continuous with respect to the parameter $k$, there exists a number $z_{{m+\frac{1}{2}},s_{0}}\in\left(\frac{j_{{m+\frac{1}{2}},s_{0}}}{\sqrt{M_1}}, \frac{j_{{m+\frac{1}{2}},s_{0}+1}}{\sqrt{M_1}}\right)$, such that
\begin{equation}
\psi_{m}(1;z_{{m+\frac{1}{2}},s_{0}})=0.
\end{equation}

Let $\beta_m=1$ in \eqref{fourier3}. For the transmission condition, it holds that
\begin{equation}
\alpha_m=\frac{\sqrt{\frac{\pi}{2}}J_{m+\frac{1}{2}}(k)}{\phi_{m}(1;k)}.
\end{equation}
Set $f_{m}=\sigma(1)\frac{\partial\phi_{m}(1;k)}{\partial r}j_{m}(k)-kj_{m}^{\prime}(k)\phi_{m}(1;k)$. Using the recursive formula of the Bessel functions \cite{CK}, we have
\begin{equation}\label{eq:fmzs0fmzs13}
f_m(k)=\sigma(1)\frac{\partial \phi_m}{\partial r}(k,1)J_{m+\frac{1}{2}}(k)+((m+1)J_{{m+\frac{1}{2}}}(k)-kJ_{m-\frac{1}{2}}(k))\phi_m(k,1),\quad m\geq 1.
\end{equation}

Next, we find the roots of $f_m(k)$ within the interval $\left(\frac{j_{{m+\frac{1}{2}},s_{0}}}{\sqrt{M_1}}, \frac{j_{{m+\frac{1}{2}},s_{0}+2}}{\sqrt{M_1}}\right)$.

\begin{lemma}\label{lem:3}
Under Assumption A, for any given $s_0 \in \mathbb{N}$, there exists $m_0(\mathbf{n},\sigma)\in \mathbb{N}$, depending on $\mathbf{n}$ and $\sigma$, such that when $m>m_0(\mathbf{n},\sigma)$, the function $f_m(k)$ in \eqref{eq:fmzs0fmzs13} possesses at least one root in $\left(\frac{j_{{m+\frac{1}{2}},s_{0}}}{\sqrt{M_1}}, \frac{j_{{m+\frac{1}{2}},s_{0}+2}}{\sqrt{M_1}}\right)$.
\end{lemma}
\begin{proof}
For any given $s_0\in\mathbb{N}$, we have for $m(s_0)\in\mathbb{N}$ large enough that
\begin{equation}
\frac{j_{{m+\frac{1}{2}},s_{0}}}{\sqrt{M_1}}<j_{{m+\frac{1}{2}},1}^{\prime}.
\end{equation}

Consider the interval $\left(\frac{j_{{m+\frac{1}{2}},s_{0}}}{\sqrt{M_1}}, \frac{j_{{m+\frac{1}{2}},s_{0}+2}}{\sqrt{M_1}}\right)$. There exist at least two consecutive zeros $z_{{m+\frac{1}{2}},s_{0}}$ and $z_{{m+\frac{1}{2}},s_{0}+1}$ of
$\psi_m(1;k)$, as well as $\phi_m(1;k)$. Using such a fact, together with the monotonicity of $J_{m+\frac{1}{2}}$ within the interval $[0,j_{{m+\frac{1}{2}},1}^{'}]$, one can show that
\begin{equation}\label{eq:Jzms0Jzms013}
\begin{array}{ll}
J_{{m+\frac{1}{2}}}(z_{{m+\frac{1}{2}},s_{0}})J_{{m+\frac{1}{2}}}(z_{{m+\frac{1}{2}},s_{0}+1})< J_{m+\frac{1}{2}}^2(j_{{m+\frac{1}{2}},1}^{'}).
\end{array}
\end{equation}

By virtue of \eqref{eq:Jzms0Jzms013} and \eqref{eq:fmzs0fmzs13}, one can show that
\begin{equation}
\begin{split}
&f_m(z_{{m+\frac{1}{2}},s_{0}})f_m(z_{{m+\frac{1}{2}},s_{0}+1})\\
=&\sigma^2(1)J_{m+\frac{1}{2}}(z_{{m+\frac{1}{2}},s_{0}})J_{m+\frac{1}{2}}(z_{{m+\frac{1}{2}},s_{0}+1})
\frac{\partial \phi_m(1;z_{{m+\frac{1}{2}},s_{0}})}{\partial r}\frac{\partial \phi_m(1;z_{{m+\frac{1}{2}},s_{0}+1})}{\partial r}\\
<& \sigma^2(1)J_{m+\frac{1}{2}}^2(j_{{m+\frac{1}{2}},1}^{\prime})\frac{\partial \phi_m(1;z_{{m+\frac{1}{2}},s_{0}})}{\partial r}\frac{\partial \phi_m(1;z_{{m+\frac{1}{2}},s_{0}+1})}{\partial r}.
\end{split}
\end{equation}
Since $z_{{m+\frac{1}{2}},s_{0}}$, $z_{{m+\frac{1}{2}},s_{0}+1}$ are two consecutive roots, it is clear that $\frac{\partial \phi_m(1;z_{{m+\frac{1}{2}},s_{0}})}{\partial r}$ and $\frac{\partial \phi_m(1;z_{{m+\frac{1}{2}},s_{0}+1})}{\partial r}$ have opposite signs. Hence,
\begin{equation}\label{eq:gg1}
f_m(z_{{m+\frac{1}{2}},s_{0}})f_m(z_{{m+\frac{1}{2}},s_{0}+1})<0.
\end{equation}
By applying Rolle's theorem to \eqref{eq:gg1}, one sees that there exists at least one root of $f_m(k)$ in $(z_{{m+\frac{1}{2}},s_{0}},z_{{m+\frac{1}{2}},s_{0}+1})$.

The proof is complete.  
\end{proof}

In the following, for a given $s_0\in\mathbb{N}$, we let $k_{l_{m}} :=k_{m,s_{0}}\in(z_{{m+\frac{1}{2}},s_{0}},z_{{m+\frac{1}{2}},s_{0}+1}),$ $m=m_0+1,m_0+2,... ,$ denote the transmission eigenvalues determined in Lemma~\ref{lem:3}, where $m_0 = m_0(\mathbf{n},\sigma)$ is chosen to be sufficiently large. 


\begin{theorem}\label{thm:surloc2}
Consider the same setup as Theorem~\ref{thm:main1} in $\mathbb{R}^3$ and assume that $\mathbf{n}>1$ and $\tau\in (0, 1)$ is fixed. Let $(w_m, v_m)$ be the pair of eigenfunctions associated with $k_{l_m}$ defined above. Then it holds that
\begin{equation}\label{eq:hh1}
\lim\limits_{m\rightarrow\infty}\frac{\|v_m\|_{L^2(\Omega_{\tau})}}{\|v_m\|_{L^2(\Omega)}}=0.
\end{equation}
\end{theorem}
\begin{proof}
Let $\beta_{m}=1$ in \eqref{eq:series1}. By direct calculations, one has that
\begin{equation*}
\left\|v_{m}\right\|_{L^{2}(\Omega_{\tau})}^{2} =\int_{\Omega_{\tau}}\left|j_{m}(k_{l_{m}}|x|)\right|^{2} \mathrm{d} x = \pi^2 \int_{0}^{\tau}rJ^{2}_{m+\frac{1}{2}}(k_{l_{m}} r),
\mathrm{d} r.
\end{equation*}
which in particular gives that
\begin{equation*}
\left\|v_{m}\right\|_{L^{2}(\Omega)}^{2} = \pi^2
\int_{0}^{1}rJ^{2}_{m+\frac{1}{2}}(k_{l_{m}} r)  \mathrm{d} r\
\text{.}
\end{equation*}
Combining
\begin{equation*}
\frac{j_{{m+\frac{1}{2}},s_0}}{\sqrt{M_1}} \leq z_{{m+\frac{1}{2}},s_{0}} \leq k_{l_{m}} \leq z_{{m+\frac{1}{2}},s_{0}+1} \leq \frac{j_{{m+\frac{1}{2}},s_0+2}}{\sqrt{M_1}} ,\quad m=m_0+1,m_0+2,...
\end{equation*}
and the estimate \eqref{eq:jms bdd}, we know that for any $s_0 \in \mathbb{N}$, there exists $m_3(M_1,\sigma)\in \mathbb{N}$ such that when $m>m_3(M_1, \sigma)$,
the following estimate holds
\begin{equation*}
\frac{1}{\sqrt{M_1}} \leq \frac{k_{l_{m}}}{m} \leq \frac{1+\sqrt{M_1}}{2\sqrt{M_1}}.
\end{equation*}

Similar to the arguments in the proof of Theorem\,2.6 in \cite{DJLZ21}, for any $\tau \in (0,1)$,  there exists $\delta(\tau, M_1, \sigma) > 0$ such that
\begin{equation}
\begin{split}
\frac{\left\|v_{m}\right\|_{L^{2}(\Omega_{\tau})}^{2}}{\left\|v_{m}\right\|_{L^{2}(\Omega)}^{2}}&\leq \frac{\int_{0}^{\tau}rJ^2_{m+\frac{1}{2}}(k_{l_{m}}r)\mathrm{d}r}
{\int_{0}^{1}rJ^2_{m+\frac{1}{2}}(k_{l_{m}}r)\mathrm{d}r} \leq \displaystyle \frac{\tau^2J^2_{m+\frac{1}{2}}(k_{l_{m}}\tau)}{\frac{\frac{1}{2}J^3_{m+\frac{1}{2}}(k_{l_{m}})}{J_{m+\frac{1}{2}}(k_{l_{m}})+2k_{l_{m}}J'_{m+\frac{1}{2}}(k_{l_{m}})}}\\
&\leq
2\tau^2\left(\frac{J_{m+\frac{1}{2}}(k_{l_{m}}\tau)}{J_{m+\frac{1}{2}}(k_{l_{m}})}\right)^2\left(1+2k_{l_{m}}\frac{J'_{m+\frac{1}{2}}(k_{l_{m}})}{J_{m+\frac{1}{2}}(k_{l_{m}})}\right)\\
&\leq
36(\sqrt{M_1}-1)m^4\tau^2\left(\frac{J_{m+\frac{1}{2}}(k_{l_{m}}\tau)}{J_{m+\frac{1}{2}}(k_{l_{m}})}\right)^2\\
&\leq144\frac{\sqrt{M_1}}{(\sqrt{M_1}-1)^2}m^4\tau^2(1-\delta(\tau,M_1,\sigma))^{2m},
\end{split}
\end{equation}
which readily gives \eqref{eq:hh1}. 

The proof is complete.
\end{proof}

\section{Proof of Theorem~\ref{thm:main12}}

This section is devoted to the proof of Theorem~\ref{thm:main12}. In what follows, we assume that $\sigma$ and $\mathbf{n}$ are both constant. We shall construct a sequence of transmission eigenvalues $k_{l_{m}}$ and prove that the corresponding eigenfunctions $(u_m , v_m)$ are both boundary-localized.

First, we know that the solutions to \eqref{eq:trans1} have the following Fourier representations:
\begin{equation}
\begin{aligned}
u(x)=\sum_{m=0}^\infty \alpha_m J_{m}(\frac{k\mathbf{n}r}{\sqrt{\sigma}}) e^{\mathrm{i}m\theta}, \quad v(x)=\sum_{m=0}^\infty \beta_m J_m(kr) e^{\mathrm{i}m\theta},\quad N=2,\\
u(x)=\sum_{m=0}^{\infty} \sum_{l=-m}^{m} \alpha_{m}^{l} j_{m}(\frac{k\mathbf{n}r}{\sqrt{\sigma}}) Y_{m}^{l}(\hat{x}),\quad v(x)=\sum_{m=0}^{\infty} \sum_{l=-m}^{m} \beta_{m}^{l}
j_{m}(k r) Y_{m}^{l}(\hat{x}),\quad N=3.
\end{aligned}
\end{equation}
Set
\begin{equation}\label{eq:bilocalized1}
\begin{aligned}
u_m(x)=\alpha_m J_{m}(\frac{k\mathbf{n}r}{\sqrt{\sigma}}) e^{\mathrm{i}m\theta}, \quad v_m(x)=\beta_m J_m(kr) e^{\mathrm{i}m\theta},\quad N=2,\\
u_m(x)=\alpha_{m} j_{m}(\frac{k\mathbf{n}r}{\sqrt{\sigma}}) Y_{m}^{l}(\hat{x}),\quad v_m(x)=\beta_{m}
j_{m}(k r) Y_{m}^{l}(\hat{x}),\quad N=3.
\end{aligned}
\end{equation}
Let $\beta_m=1$ in \eqref{eq:bilocalized1},  for the transmission condition, then set
\begin{equation}
\begin{aligned}
f^{(2)}_m(k)=J_{m}(\frac{k\mathbf{n}}{\sqrt{\sigma}})J^{\prime}_{m}(k)-\sqrt{\sigma}\mathbf{n}J^{\prime}_{m}(\frac{k\mathbf{n}}{\sqrt{\sigma}})J_{m}(k),\quad N=2,\\
f^{(3)}_m(k)=J_{m+\frac{1}{2}}(\frac{k\mathbf{n}}{\sqrt{\sigma}})j^{\prime}_{m}(k)-\sqrt{\sigma}\mathbf{n}j^{\prime}_{m}(\frac{k\mathbf{n}}{\sqrt{\sigma}})J_{m+\frac{1}{2}}(k),\quad N=3.
\end{aligned}
\end{equation}
\begin{lemma}
Suppose that $\mathbf{n}>\sqrt{\sigma}$. For any given $s_0 \in \mathbb{N}$, there exists $m_0(\mathbf{n},\sigma)\in \mathbb{N}$, depending on $\mathbf{n}$ and $\sigma$, such that when $m>m_0(\mathbf{n},\sigma)$, $f^{(2)}_m(k)$ in \eqref{eq:fmzs0fmzs13} possesses at least one root in $\left(\frac{\sqrt{\sigma}j_{{m},s_{0}}}{\mathbf{n}}, \frac{\sqrt{\sigma}j_{{m},s_{0}+1}}{\mathbf{n}}\right)$ and $f^{(3)}_m(k)$ in \eqref{eq:fmzs0fmzs13} possesses at least one root in $\left(\frac{\sqrt{\sigma}j_{{m+\frac{1}{2}},s_{0}}}{\mathbf{n}}, \frac{\sqrt{\sigma}j_{{m+\frac{1}{2}},s_{0}+1}}{\mathbf{n}}\right)$.
\end{lemma}
\begin{proof}
First, we have by direct calculations that
\begin{equation}
\begin{aligned}
&f^{(2)}_m\left(\frac{\sqrt{\sigma}j_{{m},s_{0}}}{\mathbf{n}}\right)f^{(2)}_m\left(\frac{\sqrt{\sigma}j_{{m},s_{0}+1}}{\mathbf{n}}\right)\\
=& \sigma\mathbf{n}^2J_{m-1}\left(j_{m,s_0})J_{m-1}(j_{m,s_0+1}\right)\cdot J_m\left(\frac{\sqrt{\sigma}j_{{m},s_{0}}}{\mathbf{n}}\right)J_m\left(\frac{\sqrt{\sigma}j_{{m},s_{0}+1}}{\mathbf{n}}\right),\\
&f^{(3)}_m\left(\frac{\sqrt{\sigma}j_{{m+\frac{1}{2}},s_{0}}}{\mathbf{n}}\right)f^{(3)}_m\left(\frac{\sqrt{\sigma}j_{{m+\frac{1}{2}},s_{0}+1}}{\mathbf{n}}\right)\\
=&\sigma\mathbf{n}^2J_{m-\frac{1}{2}}(j_{m+\frac{1}{2},s_0})J_{m-\frac{1}{2}}(j_{m+\frac{1}{2},s_0+1})\cdot J_{m+\frac{1}{2}}\left(\frac{\sqrt{\sigma}j_{{m},s_{0}}}{\mathbf{n}}\right)J_{m+\frac{1}{2}}\left(\frac{\sqrt{\sigma}j_{{m},s_{0}+1}}{\mathbf{n}}\right).
\end{aligned}
\end{equation}
For a fixed $s_0\in \mathbb{N}$, we have for $m(s_0)$ large enough that
\begin{equation}
\begin{aligned}
J_m\left(\frac{\sqrt{\sigma}j_{{m},s_{0}}}{\mathbf{n}}\right)J_m\left(\frac{\sqrt{\sigma}j_{{m},s_{0}+1}}{\mathbf{n}}\right)<J^2_m(j^{\prime}_{m,1}),
\\
J_{m+\frac{1}{2}}\left(\frac{\sqrt{\sigma}j_{{m},s_{0}}}{\mathbf{n}}\right)J_{m+\frac{1}{2}}\left(\frac{\sqrt{\sigma}j_{{m+\frac{1}{2}},s_{0}+1}}{\mathbf{n}}\right)<J^2_{m+\frac{1}{2}}(j^{\prime}_{m+\frac{1}{2},1}).
\end{aligned}
\end{equation}
Note that the zeros of $J_{\nu}$ are interlaced with those of $J_{\nu+1}$ for $\nu\in\mathbb{R}_{+}$ (cf. \cite{LZ}). One clearly has
\begin{equation}
f^{(2)}_m\left(\frac{\sqrt{\sigma}j_{{m},s_{0}}}{\mathbf{n}}\right)f^{(2)}_m\left(\frac{\sqrt{\sigma}j_{{m},s_{0}+1}}{\mathbf{n}}\right)<0,\ \
f^{(3)}_m\left(\frac{\sqrt{\sigma}j_{{m+\frac{1}{2}},s_{0}}}{\mathbf{n}}\right)f^{(3)}_m\left(\frac{\sqrt{\sigma}j_{{m+\frac{1}{2}},s_{0}+1}}{\mathbf{n}}\right)<0,
\end{equation}
which readily yields the desired results in the statement of the lemma by Rolle's theorem.
\end{proof}
\begin{theorem}\label{thm:bisurloc1}
Assume that $\mathbf{n}^2>\sigma$ and $\tau\in (0, 1)$ is fixed. Let $(u_m, v_m)$ be the pair of eigenfunctions associated with $k_{l_m}$ in \eqref{eq:bilocalized1}. We have
\begin{equation}\label{eq:jj1}
\lim\limits_{m\rightarrow\infty}\frac{\|u_m\|_{L^2(\Omega_{\tau})}}{\|u_m\|_{L^2(\Omega)}}=0,\quad \lim\limits_{m\rightarrow\infty}\frac{\|v_m\|_{L^2(\Omega_{\tau})}}{\|v_m\|_{L^2(\Omega)}}=0.
\end{equation}
\end{theorem}
\begin{proof}
For $v_m$, it is similar to the nonconstant case in the previous section. Hence, we only prove that $u_m$ is boundary-localized in the two dimensions and the three-dimensional case can be proved by following similar arguments.

Since $u_m(x)=J_{m}(\frac{k\mathbf{n}r}{\sqrt{\sigma}})e^{\mathrm{i}m\theta}$, one has
\begin{equation}
\frac{\mathbf{n}}{\sqrt{\sigma}}k_{l_m}>\frac{\mathbf{n}}{\sqrt{\sigma}}\frac{\sqrt{\sigma}j_{m,s_0}}{\mathbf{n}}>m(1+2(\frac{s_0+1}{m})^{\frac{2}{3}}).
\end{equation}
Next, we show that for $m$ sufficiently large, it holds that $\frac{\mathbf{n}}{\sqrt{\sigma}}k_{l_m}\tau<j^{\prime}_{m,1}$. In fact, one can first deduce that
\begin{equation}\label{inequality3}
\frac{m\sqrt{\sigma}}{\mathbf{n}k_{l_{m}}}<\frac{j^{\prime}_{m,1}\sqrt{\sigma}}{\mathbf{n}k_{l_{m}}}<\frac{m(1+3(\frac{1}{m})^{\frac{2}{3}}+2(\frac{1}{m})^{\frac{4}{3}})}{m(1+2(\frac{s_0+1}{m})^{\frac{2}{3}})},
\end{equation}
which gives that
\begin{equation}
\lim\limits_{m\rightarrow\infty}\frac{\sqrt{\sigma}j^{\prime}_{m,1}}{\mathbf{n}k_{l_{m}}}=1.
\end{equation}
Hence, for $\varepsilon=\frac{1}{2}(1-\tau)$, there exists a $m_0\in\mathbb{N}$ such that when $m>m_0$, we have
\begin{equation}
\frac{\sqrt{\sigma}j^{\prime}_{m,1}}{\mathbf{n}k_{l_{m}}}>1-\varepsilon>\frac{1}{2}(1+\tau)>\tau.
\end{equation}
That is, $\frac{\mathbf{n}}{\sqrt{\sigma}}k_{l_m}\tau<j^{\prime}_{m,1}$. Next similar to the arguments in the proof of Theorem\,2.6 in \cite{DJLZ21}, we have
\begin{equation}
\begin{array}{ll}
\frac{\left\|u_{m}\right\|_{L^{2}(\Omega_{\tau})}^{2}}{\left\|u_{m}\right\|_{L^{2}(\Omega)}^{2}}&\leq \frac{\int_{0}^{\tau}rJ^2_{m}(\frac{\mathbf{n}}{\sqrt{\sigma}}k_{l_{m}}r)\mathrm{d}r}
{\int_{0}^{1}rJ^2_{m}(\frac{\mathbf{n}}{\sqrt{\sigma}}k_{l_{m}}r)\mathrm{d}r}\\
&\leq \displaystyle \frac{\tau^2J^2_{m}(\frac{\mathbf{n}}{\sqrt{\sigma}}k_{l_{m}}\tau)}{\frac{\frac{1}{2}J^3_{m}(\frac{\mathbf{n}}{\sqrt{\sigma}}k_{l_{m}})}{J_{m}(\frac{\mathbf{n}}{\sqrt{\sigma}}k_{l_{m}})+2k_{l_{m}}J'_{m}(\frac{\mathbf{n}}{\sqrt{\sigma}}k_{l_{m}})}}\\
&\leq
2\tau^2\left(\frac{J_{m}(\frac{\mathbf{n}}{\sqrt{\sigma}}k_{l_{m}}\tau)}{J_{m}(\frac{\mathbf{n}}{\sqrt{\sigma}}k_{l_{m}})}\right)^2\left(1+2k_{l_{m}}\frac{J'_{m}(\frac{\mathbf{n}}{\sqrt{\sigma}}k_{l_{m}})}{J_{m}(\frac{\mathbf{n}}{\sqrt{\sigma}}k_{l_{m}})}\right)\\
&\leq
36(\sqrt{M_1}-1)m^4\tau^2\left(\frac{J_{m}(\frac{\mathbf{n}}{\sqrt{\sigma}}k_{l_{m}}\tau)}{J_{m}(\frac{\mathbf{n}}{\sqrt{\sigma}}k_{l_{m}})}\right)^2\\
&\leq144\frac{\sqrt{M_1}}{(\sqrt{M_1}-1)^2}m^4\tau^2(1-\delta(\tau,M_1,\sigma))^{2m}.
\end{array}
\end{equation}
which readily proves the first limit in \eqref{eq:jj1}. 

The proof is complete.

\end{proof}

\section*{Acknowledgments}
The research of H. Liu was supported by the Hong Kong RGC General Research Funds (projects 11300821, 12301420 and 12302919) and NSFC/RGC Joint Research Fund (project N\_CityU101/21). The research of J. Zhang was supported by the Natural Science Foundation of Jiangsu Province (grant no. BK20210540), the Natural Science Foundation of the Jiangsu Higher Education Institutions of China (grant no. 21KJB110015). The research of Y. Jiang and K. Zhang was supported in part by China Natural National Science Foundation (grant no. 11871245 and 11971198),  the National Key R\&D Program of China (grant no. 2020YFA0713601), and by the Key Laboratory of Symbolic Computation and Knowledge Engineering of Ministry of Education, Jilin University, China.

\end{document}